\documentclass[11pt]{article}
\usepackage{fullpage}
\usepackage{amsfonts}
\usepackage{bbding}
\usepackage{graphicx,float}
\usepackage{amsfonts,amssymb,amsmath,amsthm,amscd}
\usepackage{mathrsfs}
\usepackage{xcolor}
\usepackage{empheq}
\usepackage{adforn}
\usepackage{cancel}
\usepackage{xr-hyper}
\usepackage{xr}
\usepackage{mdframed}
\usepackage{mathdots}
\usepackage{enumerate}
\usepackage{lmodern}
\usepackage{mdframed}
\usepackage{stmaryrd}
\usepackage{blindtext}
\usepackage[all, knot]{xy}
\usepackage{leftidx}
\usepackage{accents}
\usepackage{appendix}
\usepackage{cases}
\usepackage{bbm}
\usepackage{subfigure}

\definecolor{slightblue}{rgb}{.8, .8, 1}
\definecolor{hair}{RGB}{100,225,190}
\definecolor{ruby}{RGB}{220,50,120}
\definecolor{grass}{RGB}{150,220,110}
\definecolor{ceruleanblue}{rgb}{0.16, 0.32, 0.75}
\definecolor{deepcarmine}{rgb}{0.66, 0.13, 0.24}
\definecolor{otterbrown}{rgb}{0.4, 0.26, 0.13}
\definecolor{sapphire}{rgb}{0.03, 0.15, 0.4}

\usepackage[colorlinks=true,  linkcolor=otterbrown, citecolor=sapphire,  urlcolor=hair!80!black]{hyperref}

\usepackage[T1]{fontenc}

\newtheorem{theorem}{Theorem}[section] 
\newtheorem{lemma}[theorem]{Lemma}

\theoremstyle{definition}

\newtheorem{remark}[theorem]{Remark} \numberwithin{equation}{section}
\numberwithin{figure}{section}

\newcommand{\one}{\mathbf{1}}

\newcommand{\Cb}{\mathbb{C}}

\newcommand{\Eb}{\mathbb{E}}

\newcommand{\Pb}{\mathbb{P}}

\newcommand{\Rb}{\mathbb{R}}

\newcommand{\Zb}{\mathbb{Z}}

\newcommand{\Ac}{\mathcal{A}}

\newcommand{\Fc}{\mathcal{F}}

\newcommand{\Lc}{\mathcal{L}}

\usepackage[section]{placeins}
\usepackage{wrapfig}
\usepackage{lpic}

\newcommand{\wt}{\widetilde}

\newcommand{\rwl}{\mu^{\rm drw}}
\newcommand{\br}{\mu^{\rm br}}

\title{Coupling Brownian loop soups and random walk loop soups\\ at all polynomial scales}
\author{Wei Qian \thanks{The University of Hong Kong}}
\date{}
\begin{document}
	\maketitle
	
	\begin{abstract}
Lawler and Trujillo Ferreras \cite{LTF2007} constructed a well-known coupling between the Brownian loop soups on $\Rb^2$ and the (discrete-time) random walk loop soups on $\Zb^2$ (one rescales the random walk loops by $1/N$, their time parametrizations by $1/(2N^2)$, and lets $N\to \infty$), which led to numerous applications. It nevertheless only holds for loops with time length at least $N^{\theta-2}$ for $\theta \in(2/3,2)$. 
In particular, there is no control on mesoscopic loops with time length less than $N^{-4/3}$ (i.e.\ roughly diameter less than $N^{-2/3}$). 
This coupling was subsequently extended in \cite{MR3877544} to $\Zb^d$ with $d\ge 3$, for loops with time length at least $N^{\theta-2}$, for $\theta \in(2d/(d+4),2)$.

In this paper, we find a simple way to remove the restriction $\theta>2d/(d+4)$, so that such a coupling works for all $\theta\in (0,2)$, i.e.\ for loops at all polynomial scales. 
We establish couplings for both discrete-time and continuous-time random walk loop soups on $\Zb^d$, for $d\ge 1$. As an intermediate step, we also establish a KMT coupling between the continuous-time random walk bridge on $\Zb^d$ and the Brownian bridge on $\Rb^d$.

	\end{abstract}
	
	
\section{Introduction}

The Brownian loop soup, introduced by Lawler and Werner \cite{MR2045953}, plays a prominent role in random geometry. The random walk loop soup, introduced by Lawler and Trujillo Ferreras \cite{LTF2007}, is the discrete analogue of the Brownian loop soup. 
It was proved in \cite{LTF2007} that the random walk loop soup on the square lattice $\Zb^2$ converges to the Brownian loop soup on $\Rb^2$ in the scaling limit. 
Moreover, a precise coupling between the two loop soups was constructed, as we  state below.

\subsection{Lawler and Trujillo Ferreras's coupling}\label{sec:ltf}
Let $\Ac_\lambda$ and $\wt \Ac_\lambda$ be respectively a realization of the Brownian loop soup on $\Rb^2$ and random walk loop soup on $\Zb^2$ with intensity $\lambda$ (see Section~\ref{sec:prelim} for precise definitions). Throughout this paper, we regard both the Brownian loops and random walk loops as rooted loops. We also consider the random walk loops as curves by linear interpolation. 

For each positive integer $N$, we define $\Ac_{\lambda, N}$ to be the collection of loops obtained from $\Ac_\lambda$ by scaling space by $1/N$ and time by $1/N^2$. More precisely, for each loop $\gamma$ with time length $t_\gamma$, let $\Phi_N\gamma$ be the loop with 
\begin{equation}\label{eq:phin}
t_{\Phi_N \gamma} =t_\gamma/N^2, \qquad \Phi_N\gamma(t) =N^{-1} \gamma(t N^2), \quad 0\le t \le t_\gamma/N^2.
\end{equation}
We then set
\begin{align}\label{eq:ac}
\Ac_{\lambda, N} =\{\Phi_N\gamma : \gamma \in \Ac_\lambda\}.
\end{align}
The scaling invariance of the Brownian loop soup implies that $\Ac_{\lambda, N}$ is also distributed as a Brownian loop soup with intensity $\lambda$.

For the random walk loop soup, we define $\wt\Phi_N \gamma$ to be the loop with 
\begin{align}\label{eq:phi2d}
t_{\wt\Phi_N \gamma} =t_\gamma/(2 N^2), \qquad \wt\Phi_N\gamma(t) =N^{-1} \gamma(t 2 N^2), \quad 0\le t \le t_\gamma/ (2N^2),
\end{align}
and let
\begin{align}\label{eq:a2d}
\wt \Ac_{\lambda, N} =\{\wt \Phi_N\gamma : \gamma \in \wt \Ac_\lambda\}.
\end{align}
As pointed out in \cite{LTF2007}, the extra factor $2$ in the time scaling for the random walk loops comes from the fact that the covariance of the simple two-dimensional random walk in $2n$ steps is $n I$ as opposed to $2n I$ for a Brownian motion at time $2n$.

To state the coupling by Lawler and Trujillo Ferreras, we also need to define the following auxiliary functions. 
For $t\ge (5/8) N^{-2}$ and integer $k\ge 1$, let
\begin{align*}
\varphi_N(t) = \frac{k}{N^2} \quad \text{if} \quad \frac{k-3/8}{N^2} \le t < \frac{k+5/8}{N^2}.
\end{align*}
For $z\in \Cb$ and $z_0 \in \Zb^2 \subset \Cb$, define
\begin{align}\label{eq:psi}
\psi_N(z) =\frac{z_0}{N} \quad \text{if} \quad \max\{|\Re\{Nz -z_0\}|, |\Im\{Nz-z_0\}|\} <1/2.
\end{align}
The definition of $\psi_N(z)$ when $Nz$ is exactly on an edge of the dual lattice of $\Zb^2$  is irrelevant for the following theorem.
\begin{theorem}[Theorem 1.1, \cite{LTF2007}]\label{thm0}
One can define $\{\Ac_\lambda\}_{\lambda>0}$ and $\{\wt \Ac_\lambda\}_{\lambda>0}$ on the same probability space so that 
both $\Ac_\lambda$ and $\wt\Ac_\lambda$ are increasing in $\lambda$.
Moreover, for $\Ac_{\lambda, N}$ and $\wt \Ac_{\lambda,N}$ defined as above, there exists $c>0$ such that for every $r\ge 1$, $\lambda>0$, integer $N\ge 1$ and every $\theta \in (2/3, 2)$, except perhaps on an event of probability at most $c(\lambda +1) r^2 N^{2-3\theta}$, there is a one-to-one correspondence between the following two sets of loops
\begin{align*}
\{\wt \gamma\in \wt \Ac_{\lambda,N} : t_{\wt \gamma} > N^{\theta-2}, |\wt \gamma(0)|<r\}, \quad \{\gamma\in \Ac_{\lambda,N} : \varphi_N(t_{ \gamma}) > N^{\theta-2}, |\psi_N( \gamma(0))|<r\}.
\end{align*}
If $\wt\gamma\in \wt\Ac_{\lambda, N}$ and $\gamma\in\Ac_{\lambda,N}$ are paired in this correspondence, then
\begin{align}
\notag
&|t_\gamma - t_{\wt \gamma}| \le 5/8 N^{-2}, \qquad\\
\label{eq:ltf}
\sup_{0\le s \le 1}& |\gamma(s t_\gamma) -\wt\gamma(s t_{\wt \gamma})| \le c N^{-1} \log N.
\end{align}
\end{theorem}

The coupling by Lawler and Trujillo Ferreras has proved useful in a great number of situations. 
First, it provides the theoretical basis for simulations of Brownian loop soups via random walk loop soups. Perhaps the best-known application of this coupling is the proof of  the convergence of random walk loop soup clusters towards the conformal loop ensemble (CLE): it was crucially employed in \cite{BCL2016} which, together with \cite{Lu2019}, gives the complete proof. 
In many cases, the random walk loop soup has been instrumental for proving results about the Brownian loop soup in the continuum, see e.g.\ \cite{QW2019,ALS2020b,ABJL2023}.
Conversely, CLE arm exponents \cite{GNQ3} were used in \cite{GNQ1} to compute arm exponents of the random walk loop soup, which served as a key input to study percolation of the discrete Gaussian free field (GFF) and the occupation field of the random walk loop soup \cite{GNQ2}. 

Nevertheless, since we can only take $\theta>2/3$ in Theorem~\ref{thm0}, this coupling does not control mesoscopic loops with time length less than $N^{-4/3}$ (which roughly corresponds to diameter less than $N^{-2/3}$).
In practice, it is often necessary to control the mesoscopic loops, since they can potentially play a role in the scaling limit in terms of connectivity.
For instance, the restriction $\theta>2/3$ was problematic for the purpose of A\"id\'ekon, Berestycki, Jego and Lupu \cite{ABJL2023}, when approximating the multiplicative chaos of the Brownian loop soup by that of the random walk loop soup (see Sections 1.5 and 11.3 in \cite{ABJL2023} for more details).
In \cite[Lemma 11.11]{ABJL2023}, they constructed a coupling for mesoscopic loops at all polynomial scales, with the drawback that it only works for \emph{local} loops, i.e.\ the loops whose roots are close to a given point $z$ (roughly, the distance of the root of a loop to $z$ is at most the order of the diameter of that loop).

\subsection{Main result}\label{subsec:main}
It is unclear at first sight whether $\theta>2/3$ can be removed,\footnote{
For example, this was believed to be impossible in \cite[Section 11.3]{ABJL2023}: 
``On the other hand, it is fairly clear from the proof of \cite{LTF2007} that their result is sharp, and that the coupling described above cannot hold without the restriction $\theta>2/3$; that is, at all scales smaller than $N^{-2/3}$ some discrete and continuous loops somewhere will be quite different from one another.''}  since the value $2/3$ comes from a computation involving the following Taylor expansion of the total mass $q_2(n)$ of all random walk loops on $\Zb^2$ with length $2n$ rooted at $0$
\begin{align}\label{eq:qn}
\wt q_2(n)= (2n)^{-1} \left[ 2^{-2n}\binom{2n}{n} \right]^2 = \frac{1}{2\pi n^2} - \frac{1}{8 \pi n^3} + O(n^{-4}).
\end{align}
Somewhat unexpectedly, after a close study of the proof by Lawler and Trujillo Ferreras \cite{LTF2007},  we have found a natural way to get rid of the constraint $\theta>2/3$. 
This shows that the value $2/3$ in fact does not play an inherent role, and the size $N^{-2/3}$ is not an impassable barrier.
Our proof follows the outline of \cite{LTF2007}, by carefully tuning some parameters. Our observation is rather simple, but it seems that it has been overlooked.

Lawler and Trujillo Ferreras established this coupling in dimension $2$, but their proof has been extended to dimensions $d\ge 3$ by Sapozhnikov and Shiraishi in \cite[Theorem 2.2]{MR3877544}, relying on the expansion (see e.g.\ \cite{MR2137045})
\begin{align}\label{eq:qndd}
\wt q_d(n)= \frac{1}{n}\bigg(\frac{d}{4\pi n} \bigg)^{d/2} \bigg(1- \frac{d}{8n} +O\big(\frac{1}{n^2}\big) \bigg),
\end{align}
where  $\wt q_d(n)$ is the total mass of all random walk loops on $\Zb^d$ with length $2n$ rooted at $0$.
Similarly, the coupling in \cite{MR3877544} only works for loops with length greater than $N^{\theta-2}$, where 
\begin{equation}\label{eq:theta}
2d/(d+4)< \theta <2,
\end{equation}
on an event with probability at least 
\begin{equation}\label{eq:bad_event}
1- c(\lambda+1) r^d N^{-\min \left(\frac{d}{2},\; \theta \left(\frac{d}{2}+2\right)-d \right)}
\end{equation}
for some $c\in(0,\infty)$.
They also obtained a slightly weaker bound $c N^{-1/4}\log N$ on the supremum distance between the coupled loops, 
instead of $c N^{-1}\log N$ in \eqref{eq:ltf}.
This is because they established a weaker coupling between random walk bridges on $\Zb^d$ and the Brownian bridges (see Lemma~\ref{lem:kmtd}), for $d\ge 3$, compared to
the strong KMT coupling for $d=2$ established in \cite{LTF2007} (see Lemma~\ref{lem:kmt}). This strong coupling, as well as the exact combinatorial formula \eqref{eq:qn}, relies on a trick which is specific to $d=2$: If we consider a simple random walk on $\Zb^2$ rotated by $45$ degrees, then its projections to the vertical and horizontal axes are independent simple random walks. 
\smallskip

In this paper, we consider both the \emph{discrete-time} random walk loop soups (as in \cite{LTF2007, MR3877544}), and the \emph{continuous-time} random walk loop soups on $\Zb^d$, where every walk spends an (independent) exponential time with mean $1$ at each vertex. 
By abuse of notation, we use $\wt \Ac_\lambda$ to denote both versions of random walk loop soups. The actual version we use will be specified in the context.
The two versions of random walk loop soups are the same, if one forgets the time parametrization and remove the single-vertex loops in the continuous-time random walk loop soup (see \cite[Section 9]{MR2677157}).
However, since our coupling takes into account the time-parametrization of the loops, the couplings for the two versions of loop soups need separate proofs, which can mostly be done in parallel (also see Remark~\ref{rmk:time}). 

The continuous-time random walk loop soup is more natural when one considers its occupation-time field, because of a well-known isomorphism theorem relating this field to the GFF \cite{MR2815763}. 
In addition, it is equal to the ``restriction'' of the Brownian loop soup on the metric graph of $\Zb^d$ (i.e.\ we restrict the Brownian loops on the metric graph to stay on a vertex, before jumping to the next vertex). In other words, we can naturally couple these two versions of loop soups to establish a one-to-one correspondence between them, with the paired loops staying within a supremum distance of 1 of each other.

The Brownian loop soup on metric graphs, introduced by Lupu  \cite{MR3502602}, has proved to be a powerful tool for connecting discrete and continuum models. It has since developed into an active area of research, see e.g.\  \cite{MR3502602,lupu2018random, QW2019, ALS2020b, CD2024, drewitz2025critical, CD20242, Werner2025, LW2025}. Despite the extensive literature on this subject, a KMT-type coupling between the continuous-time random walk bridge on $\Zb^d$ (or equivalently the Brownian bridge on the metric graph of $\Zb^d$) and the Brownian bridge on $\Rb^d$ has not been previously proved (the bridges can be seen as rooted loops).
As an important proof ingredient, we establish such a strong KMT coupling  for all $d\ge 1$. 
\begin{theorem}\label{lem:kmt3}
Fix $d\ge 1$. For each $\eta>0$, there exists $c(\eta, d) \in (0,\infty)$ so that the following holds. For every  $t>0$, there exists a probability space $(\Omega, \Fc, \Pb)$ on which are defined a $d$-dimensional Brownian bridge $(B_s, 0\le s\le 1)$ and a $d$-dimensional (continuous-time) random walk bridge $(S_s, 0\le s\le t)$ such that
\begin{align*}
\Pb\left[ \sup_{0\le s\le 1} \left| (t/d)^{-1/2} S_{st} - B_s \right | \ge c(\eta, d) t^{-1/2} \log t \right] \le c(\eta, d) t^{-\eta}.
\end{align*}
\end{theorem}

Another ingredient that we rely on is the following Taylor expansion of the total mass $q_d(t)$ of all continuous-time random walk loops on $\Zb^d$ with time length $t$ rooted at $0$ (see \eqref{eq:qdt_def} and~\eqref{eq:pt0})
\begin{align}\label{eqnd}
q_d(t)= (d/2)^{d/2} \pi^{-d/2} t^{-d/2-1} + O\big(t^{-d/2-2}\big).
\end{align}
Our proof  requires only the first term in the Taylor series, unlike in \cite{LTF2007, MR3877544}.

\smallskip

In order to state our result, we keep the same notations as in Section~\ref{sec:ltf}, with the following obvious extensions. For $d\ge 1$, 
for a simple random walk loop $\gamma$ on $\Zb^d$, let  $\wt\Phi_N \gamma$ be the loop with
\begin{align*}
t_{\wt\Phi_N \gamma} =t_\gamma/(dN^2), \qquad \wt\Phi_N\gamma(t) =N^{-1} \gamma(t  dN^2), \quad 0\le t \le t_\gamma/ (dN^2).
\end{align*}
For $z\in \Rb^d$ and $z_0\in \Zb^d$,  let $\psi_N(z)=z_0/N$, if $Nz$ is in the unit hypercube centered at $z_0$. The definition of $\psi_N(z)$ when $Nz$ falls on the boundary of any unit hypercube centered at a vertex of $\Zb^d$ is irrelevant for our theorem.

The main new input in our proof is the introduction of a positive increasing sequence $\{a_n\}_{n\ge1}$, whose precise definition is given in \eqref{eq:an}. We also show (in Lemma~\ref{lem:key}) that there exists a constant $c(d)>0$, such that 
\begin{align}\label{eq:c'}
|a_n - 2n/d | \le c(d) .
\end{align}
For $t\ge a_1 N^{-2}$ and integer $k\ge 1$, let
\begin{align}\label{eq:chi1}
\chi_N(t) =k \quad \text{if} \quad \frac{a_k}{N^2} \le t < \frac{a_{k+1}}{N^2}.
\end{align}

\begin{theorem}\label{thm1}
Fix $d\ge 1$. Let $\Ac_\lambda$ be a Brownian loop soup on $\Rb^d$, and $\wt\Ac_\lambda$ be a (continuous-time) random walk loop soup on $\Zb^d$.
One can define $\{\Ac_\lambda\}_{\lambda>0}$ and $\{\wt \Ac_\lambda\}_{\lambda>0}$ on the same probability space so that 
both $\Ac_\lambda$ and $\wt\Ac_\lambda$ are increasing in $\lambda$.

Moreover,   
for every $a>0$ and every $\theta \in (0,2)$, there exists $c>0$ such that for every $r\ge 1, \lambda>0$ and integer $N\ge 1$, except perhaps on an event of probability at most $c \lambda r^d N^{-a}$, there is a one-to-one correspondence between 
$$\{\wt \gamma\in \wt \Ac_{\lambda,N} : \lfloor t_{\wt \gamma} dN^2/2 \rfloor> N^{\theta}, |\wt \gamma(0)|<r\}, \quad \{\gamma\in \Ac_{\lambda,N} : \chi_N(t_{ \gamma}) > N^{\theta}, |\psi_N( \gamma(0))|<r\}.$$ 
If $\wt\gamma\in \wt\Ac_{\lambda, N}$ and $\gamma\in\Ac_{\lambda,N}$ are paired in this correspondence, then for the constant $c(d)$ in \eqref{eq:c'},
\begin{align}
\label{c1}
&|t_\gamma - t_{\wt \gamma}| \le c(d) N^{-2}, \qquad\\[1mm] 
\label{c2}
\sup_{0\le s \le 1}& |\gamma(s t_\gamma) -\wt\gamma(s t_{\wt \gamma})| \le c N^{-1} \log N.
\end{align}
\end{theorem}
Note that Theorem~\ref{thm1} works for loops at any polynomial scale, rooted anywhere in a macroscopic domain. Moreover, the error probability (of the event on which the coupling fails) is at most $c\lambda r^d N^{-a}$, where $a$ can be made arbitrarily large, compared with Theorem~\ref{thm0} and \eqref{eq:bad_event}.
This strengthened coupling can have many potential applications. For example, as we mentioned earlier, it was proved by \cite{Lu2019, BCL2016} that in dimension two, the outer boundaries of the outermost clusters of the random walk loop soups with intensity $\lambda \in (0,1/2]$ converge to those of the Brownian loop soups (which are in turn distributed as CLE \cite{MR2979861}). This convergence is nevertheless qualitative. If one wants to obtain a precise convergence rate, it then seems necessary to control the random walk loops at mesoscopic scales, everywhere in the domain simultaneously, since the outer boundary of a cluster is a global object.

\smallskip
The result for discrete-time random walk loop soups on $\Zb^d$ is rather similar, but we still state it below in Theorem~\ref{thm2}, so that the reader can directly compare it with \cite{LTF2007} in the $d=2$ case (and with  \cite{MR3877544} in the $d\ge 3$ case). 
In the discrete-time case, we use a different sequence $\{a_n\}_{n\ge1}$ defined in \eqref{eq:an12}, which also satisfies \eqref{eq:c'}. We keep the same definition \eqref{eq:chi1} of $\chi_N$, but inputting the new sequence $\{a_n\}_{n\ge1}$.

\begin{theorem}\label{thm2}
Fix $d\ge 1$. Let $\Ac_\lambda$ be a Brownian loop soup on $\Rb^d$, and $\wt\Ac_\lambda$ be a (discrete-time) random walk loop soup on $\Zb^d$.
One can define $\{\Ac_\lambda\}_{\lambda>0}$ and $\{\wt \Ac_\lambda\}_{\lambda>0}$ on the same probability space so that 
both $\Ac_\lambda$ and $\wt\Ac_\lambda$ are increasing in $\lambda$.

Moreover,   
for every $a>0$ and every $\theta \in (0,2)$, there exists $c>0$ such that for every $r\ge 1, \lambda>0$ and integer $N\ge 1$, except perhaps on an event of probability at most $c \lambda r^d N^{-a}$, there is a one-to-one correspondence between 
$$\{\wt \gamma\in \wt \Ac_{\lambda,N} :  t_{\wt \gamma} dN^2/2 > N^{\theta}, |\wt \gamma(0)|<r\}, \quad \{\gamma\in \Ac_{\lambda,N} : \chi_N(t_{ \gamma}) > N^{\theta}, |\psi_N( \gamma(0))|<r\}.$$ 
If $\wt\gamma\in \wt\Ac_{\lambda, N}$ and $\gamma\in\Ac_{\lambda,N}$ are paired in this correspondence, then for the constant $c(d)$ in \eqref{eq:c'},
\begin{align}
\label{c11}
&|t_\gamma - t_{\wt \gamma}| \le c(d) N^{-2}, \qquad\\[1mm] 
\label{c22}
\sup_{0\le s \le 1}& |\gamma(s t_\gamma) -\wt\gamma(s t_{\wt \gamma})| \le c N^{-1} \log N \qquad \text{ if } d=1,2,\\
\label{c33}
\sup_{0\le s \le 1}& |\gamma(s t_\gamma) -\wt\gamma(s t_{\wt \gamma})| \le c N^{\frac{a-d}{2d}} \log N \quad\; \text{ if } d\ge 3.
\end{align}
\end{theorem}

\begin{remark}
For $d\ge 3$, we get the bound $c N^{\frac{a-d}{2d}} \log N$ in \eqref{c33}, using the coupling between the random walk bridge and the Brownian bridge (Lemma~\ref{lem:kmtd}) established in  \cite{MR3877544}. 
This bound tends to $0$ as $N\to\infty$ only when $a<d$, hence we cannot reduce the error probability (of the event on which the coupling fails) below $c\lambda r^d N^{-d}$ in this case.
Note that this does not prevent us from choosing $\theta$ close to $0$, hence this coupling does still control loops at all polynomial scales. 
If we choose $a=d/2$, then we get the same $cN^{-1/4} \log N$ bound as in  \cite{MR3877544}. 

We believe that it is possible to improve the bound \eqref{c33}, by improving the coupling of Lemma~\ref{lem:kmtd}. However, if we do not care about the time parametrization, for example if we only care about the Hausdorff distance between the coupled loops, then \eqref{c2} already provides the strong bound $c N^{-1} \log N$.
\end{remark}

\begin{remark}\label{rmk:time}
An immediate corollary of Theorems~\ref{thm1} and~\ref{thm2} is that the discrete-time and continuous-time random walk loop soups can also be coupled together, since they can both be coupled with the Brownian loop soup. This coupling is, however, quite different from the obvious coupling where we let the two versions of loop soups have the same trace on $\Zb^d$ (except for the single-vertex loops). Note that two random walk loops with the same trace, respectively parametrized by discrete and continuous time, typically do not have time lengths close to each other in the sense of \eqref{c11}.
\end{remark}

\section{Preliminaries}\label{sec:prelim}
In this section, we recall the definitions and some basic properties of the Brownian loop soup, the discrete-time and continuous-time random walk loop soups on $\Zb^d$, and record several versions of couplings between Brownian bridges and random walk bridges.
\subsection{Brownian loop soup}
The Brownian loop soup on $\Rb^d$ with intensity $\lambda>0$ is a Poisson point process with intensity $\lambda \br$, where $\br$ is a $\sigma$-finite measure on Brownian loops defined as follows
\begin{align}\label{eq:bl}
\br=\int_{\Rb^d}\int_0^\infty \frac{1}{t (2\pi t)^{d/2}} \Pb^{\rm br}(z,z;t) dt dz,
\end{align}
where $dz$ is according to the Lebesgue measure on $\Rb^d$, and $\Pb^{\rm br}(z,z;t)$ is the probability measure of a $d$-dimensional Brownian bridge (which is a loop) starting and ending at $z$ with time length $t$. The point $z$ is also called the \emph{root} of the loop.

It is common to consider the loops in the Brownian loop soup as unrooted loops, by forgetting their roots. 
However, for the purpose of this paper, we will always regard the Brownian loops as rooted loops, which makes it easier to compare the difference between two parametrized loops.

It was shown in \cite{MR2045953} that the two-dimensional Brownian loop soup (as a set of unrooted loops) is invariant under conformal maps (up to time reparametrization). In this paper, we will use the fact that the $d$-dimensional Brownian loop soup (as a set of rooted loops) is invariant under the usual Brownian scaling (i.e.\ we scale space by $c$ and time by $c^2$), which is obvious from \eqref{eq:bl}.

\subsection{Discrete-time random walk loop soup}
A \emph{rooted discrete-time} random walk loop on the  lattice $\Zb^d$ is a path $(z_1, \cdots, z_j)$ in $\Zb^d$ such that $|z_i - z_{i+1}| =1$ for all $1\le i \le j-1$ and $z_j =z_1$. Let $t_\gamma:=j-1$ be the length of $\gamma$.  The point $z_1=z_j$ is called the root of the loop. One can also unroot the loops on $\Zb^d$, but we will only look at rooted loops in this paper.

Let $\rwl$ be the random walk loop measure which assigns weight $(2n)^{-1} (2d)^{-2n}$ to each rooted loop on $\Zb^d$ with length $2n$, for any integer $n\ge 1$.
For $z\in\Zb^d$ and integer $n\ge 1$, let $\nu_{n,z}$ be the measure $\rwl$ restricted to loops rooted at $z$ with length $2n$. Then we have
\begin{align*}
\rwl=\sum_{z\in \Zb^d} \sum_{n\ge 1} \nu_{n,z}.
\end{align*}
The (discrete-time) random walk loop soup with intensity $\lambda$ is a Poisson point process with intensity $\lambda \rwl$, which is a multi-set of random walk loops.
Let (note that $|\nu_{n,z}|$ is the same for all $z$)
\begin{align}\label{eq:qn1}
\wt q_d(n):=|\nu_{n,z}|.
\end{align}
The combinatorial formula of $\wt q_2(n)$ is given in  \eqref{eq:qn} (and a similar formula can easily be obtained for $\wt q_1(n)$).
In general, the Taylor expansion of $\wt q_d(n)$ is given by \eqref{eq:qndd}.

\subsection{Continuous-time random walk loop soup}
Consider the continuous-time random walk on $\Zb^d$ which spends an exponential time with mean $1$ at each vertex, before jumping to one of the neighboring vertices with probability $1/(2d)$. For $x,y\in\Zb^d$ and $t>0$, let $p_t(x,y)$ be its transition kernel. Let $\Pb^{\rm  crw}(x,y;t)$ be the probability measure of a continuous-time random walk bridge with time length $t$ starting at $x$ and ending at $y$.
Let $\mu^{\rm crw}$ be the measure on random walk loops defined by
\begin{align*}
\mu^{\rm crw}=\sum_{z\in \Zb^d} \int_0^\infty \frac{1}{t} p_t(0,0) \Pb^{\rm crw}(z,z;t) dt.
\end{align*}
The (continuous-time) random walk loop soup with intensity $\lambda$ is a Poisson point process with intensity $\lambda \mu^{\rm crw}$. Let 
\begin{equation}\label{eq:qdt_def}
q_d(t):=\frac{1}{t} p_t(0,0).
\end{equation}
Thanks to the following Taylor development of $p_t(0,0)$ (see e.g.\ \cite[Theorem 2.1.3]{MR2677157})
\begin{align}\label{eq:pt0}
p_t(0,0)=\bigg(\frac{d}{2\pi t} \bigg)^{d/2} (1+O(t^{-1})),
\end{align}
we get the Taylor development \eqref{eqnd} of $q_d(t)$.
We also define the following quantity
\begin{align}\label{eq:qdn}
Q_d(n):=\int_{2n}^{2n+2} q_d(t) dt.
\end{align}

\subsection{Coupling Brownian bridge and discrete-time random walk bridge}
In this subsection, we review some known results regarding the coupling of Brownian bridges and discrete-time random walk bridges.
We regard the random walk bridges as continuous curves, by linear interpolation.

First of all, we need the following coupling between a discrete-time simple random walk bridge on $\Zb$ and a Brownian bridge with length $2n$, established in \cite[Theorem 6.4]{LTF2007}.  
\begin{theorem}[Theorem 6.4, \cite{LTF2007}]
\label{thm:kmt1}
There exist $0<c, \alpha <\infty$, such that for every positive integer $n$, there is a probability space on which are defined a one-dimensional Brownian bridge $B$ and a (discrete-time) random walk bridge $S$ on $\Zb$, both with length $2n$, such that for all $r>0$, 
\begin{align*}
\Pb\left[ \sup_{0\le t \le 2n} |B_t - S_t| > r c \log n \right] \le c n^{\alpha -r}.
\end{align*}
\end{theorem}

By the trick of rotating $\Zb^2$ by $45$ degrees (see the second paragraph of Section~\ref{subsec:main}), the simple random walk on $\Zb^2$ is the combination of two independent simple random walks on $\Zb$, hence Theorem~\ref{thm:kmt1} implies the following coupling (see \cite[Lemma 3.1]{LTF2007} and the remark just below it).
\begin{lemma}\label{lem:kmt}
There exists a constant $c_0>0$, such that for each $\eta>0$, there exists $c(\eta) \in (0,\infty)$ so that the following holds. For every positive integer $n$, there exists a probability space $(\Omega, \Fc, \Pb)$ on which are defined a two-dimensional Brownian bridge $(B_t, 0\le t\le 1)$ and a two-dimensional (discrete-time) random walk bridge $(S_t, 0\le t\le 2n)$ such that
\begin{align}\label{eq:lawler}
\Pb\left[ \sup_{0\le s\le 1} \left| n^{-1/2} S_{2ns} - B_s \right | \ge c(\eta) n^{-1/2} \log n \right] \le c_0 n^{-\eta}.
\end{align}
\end{lemma}

In \cite[Lemma 4.2]{MR3877544}, the authors established the following coupling between random walk bridges on $\Zb^d$ and $d$-dimensional Brownian bridges, for $d\ge 3$. They got a weaker $n^{-1/4}$ factor in \eqref{eq:ss}, compared with $n^{-1/2}$ in \eqref{eq:lawler}, due to the technical difficulty arising from the fact that the movements of the random walk in $d$ coordinates are not independent.
\begin{lemma}[Lemma 4.2, \cite{MR3877544}]
\label{lem:kmtd}
Fix $d\ge 3$. For any $\eta>0$,  there exists $c(\eta,d) \in (0,\infty)$ so that the following holds. For every positive integer $n$, there exists a probability space $(\Omega, \Fc, \Pb)$ on which are defined a $d$-dimensional Brownian bridge $(B_t, 0\le t\le 1)$ and a (discrete-time) random walk bridge $(S_t, 0\le t\le 2n)$ on $\Zb^d$ such that
\begin{align}\label{eq:ss}
\Pb\left[\sup_{0\le s\le 1} \big| (2n/d)^{-1/2} S_{2ns} -B_s\big| > c(\eta,d) n^{-1/4} \log n \right] \le c(\eta,d) n^{-\eta}.
\end{align}
\end{lemma}

\section{Coupling Brownian bridge and continuous-time random walk bridge}

The goal of this section is to prove the KMT coupling stated in Theorem~\ref{lem:kmt3}.
Our proof relies on the following result \cite[Theorem 3]{MR375412} from the original paper by Koml\'os, Major, and Tusn\'ady, and is otherwise self-contained.
\begin{theorem}[Theorem 3, \cite{MR375412}]\label{thm3.1}
Fix $n\ge 1$. Let $X_1, \ldots, X_n$ be i.i.d.\ uniform random variables on $[0,1]$. Let $F_n(t):=\sum_{k=1}^n \one_{X_k\le t}$. There is a version of $F_n(t)$ and a Brownian bridge $(B^n_t, 0\le t \le 1)$ with time length $1$ such that
\begin{align*}
\Pb\left(\sup_{0\le t\le 1} |F_n(t) -nt - n^{1/2} B^n_t| > C \log n +x\right) <K e^{-\lambda x},
\end{align*}
for all $x>0$, where $C, K, \lambda$ are positive absolute constants.
\end{theorem}

We first prove the coupling in dimension $d=1$, stated below.

\begin{lemma}\label{thm:kmt2}
For any $\eta>0$, there exists $c(\eta) \in (0,\infty)$ so that the following holds.
For all $t>0$, there is a probability space on which are defined a one-dimensional Brownian bridge $(B_s, 0\le s\le 1)$ and a (continuous-time) random walk bridge $(S_s, 0\le s\le t)$ on $\Zb$ such that 
\begin{align*}
\Pb\left[ \sup_{0\le s \le 1} | t^{-1/2} S_{st} - B_s| > c(\eta) t^{-1/2}\log t \right] \le c(\eta) t^{-\eta}.
\end{align*}
\end{lemma}

\begin{proof}
Suppose that $S$ has time length $t$ and makes a total of $2M$ steps. 
For each $m>0$,
on the event $\{M=m\}$, since the times at which the $2m$ steps of $S$ occur are uniformly distributed in $[0,t]$, we can write for $s\in[0,t]$
\begin{align*}
S_s=\sum_{i=1}^m \one_{U_i\le s} - \sum_{i=1}^m \one_{V_i\le s},
\end{align*}
where $U_1, \ldots, U_m$ and $V_1, \ldots, V_m$ are $2m$ independent and uniform random variables in $[0,t]$.
By Theorem~\ref{thm3.1}, there exist two independent Brownian bridges $B^{m,1}$ and $B^{m,2}$ with time length $1$ and absolute constants  $c_1, c_2, \lambda>0$ such that for all $x>0$
\begin{align}
&\Pb \bigg[ \sup_{0\le s\le 1} \big |\sum_{i=1}^m \one_{U_i\le st} - m s - \sqrt{m} B^{m,1}_s \big| > c_1 \log m +x \bigg] < c_2 \exp(-\lambda x),\\
&\Pb \bigg[ \sup_{0\le s\le 1} \big|\sum_{i=1}^m \one_{V_i\le st} - m s - \sqrt{m} B^{m,2}_s \big| > c_1 \log m +x \bigg] < c_2 \exp(-\lambda x).
\end{align}
Let $B^m:= (B^{m,1} - B^{m,2})/\sqrt{2}$. Note that $B^m$ is also a Brownian bridge with time length $1$. We have
\begin{align}\label{eq:B_m}
\Pb \bigg[ \sup_{0\le s\le 1} \big |S_{st} - \sqrt{2m} B^m_s \big| > 2c_1\log m +2x \mid M=m \bigg] < 2c_2\exp(-\lambda x).
\end{align}
Since \eqref{eq:B_m} is true for all $m\ge 1$, we have
\begin{align}\label{eq:B_M}
\Pb \bigg[ \sup_{0\le s\le 1} \big |S_{st} - \sqrt{2M} B^M_s \big| > 2c_1\log M +2x \mid M>0 \bigg] < 2c_2\exp(-\lambda x).
\end{align}
Note that it is possible to make the Brownian bridges $(B^m, m> 0)$ independent from each other and independent from $M$. 
We also define the Brownian bridge $B^0$ with time length $1$, independent from everything else. In this way, $B^M$ is distributed as a Brownian bridge with time length $1$, whether we condition on $M>0$ or not.

The distribution of $M$ can be derived explicitly. Note that the (continuous-time) random walk bridge $S$ with length $t$ is distributed as a (continuous-time) random walk $X$ conditioned on $X_t=0$. For any integer $m\ge 0$, the probability that $X$ makes exactly $m$ upward jumps and $m$ downward jumps on $[0,t]$ is given by $e^{-t} \frac{(t/2)^{2m}}{(m!)^2}$. This implies that
\begin{align*}
\Pb(X_t=0) = \sum_{m=0}^\infty e^{-t} \frac{(t/2)^{2m}}{(m!)^2}=e^{-t} I_0(t),
\end{align*}
where $I_0(t) = \sum_{m=0}^\infty \frac{(t/2)^{2m}}{(m!)^2}$ is the modified Bessel function of the first kind of order zero. This implies
\begin{align}\label{eq:dist_M}
\Pb(M=m)=\Pb(X \text{ makes } 2m \text{ jumps on } [0,t] \mid X_t=0) =\frac{(t/2)^{2m}}{(m!)^2 I_0(t)}.
\end{align}
Let us now show that $M$ is concentrated around $t/2$. More precisely, we show that there exists $c_3>0$ such that for all $0\le u \le  \sqrt{t}$,
\begin{align}\label{eq:concentration}
\Pb(|2M- t| > u\sqrt{t}) \le c_3 e^{-u^2/4}.
\end{align}
By \eqref{eq:dist_M}, we have
\begin{align*}
\Eb[\exp(\theta (2M-t))] = e^{-\theta t} \frac{I_0(t e^\theta)}{I_0(t)}.
\end{align*}
Since $I_0(x) \asymp x^{-1/2} e^x$ as $x\to \infty$, there exist $t_0>0$ and $c_4>0$ such that for all $t\ge t_0$ and  $|\theta|\le 1/2$,
\begin{align*}
\Eb[\exp(\theta (2M-t))] \le c_4 e^{-\theta t} e^{t (e^\theta -1) -\theta/2}.
\end{align*}
For $|\theta|\le 1/2$, we have $e^\theta-1-\theta\le \theta^2$, hence there is $c_5>0$ such that for all $t\ge t_0$
\begin{align*}
\Eb[\exp(\theta (2M-t))] \le c_5 e^{ t \theta^2}.
\end{align*}
Fix $0\le u\le \sqrt{t}$. It follows that for $\theta \in [0, 1/2]$ and $t\ge t_0$
\begin{align*}
\Pb(2M-t > u\sqrt{t}) \le e^{-\theta u \sqrt{t}} \Eb[e^{\theta(2M-t)}] \le c_5 e^{ t\theta^2 -\theta u\sqrt{t}}.
\end{align*}
We can choose $\theta = \frac{u}{2 \sqrt{t}} \in [0, 1/2]$, so that for $t\ge t_0$
\begin{align}\label{eq:conc1}
\Pb(2M-t > u\sqrt{t}) \le c_5 e^{-u^2/4}.
\end{align}
By adjusting $c_5$, one can relax the condition $t\ge t_0$ to all $t>0$ in \eqref{eq:conc1}.
Similarly, using the inequality $\Pb(t- 2M> u\sqrt{t}) \le e^{-\theta u \sqrt{t}} \Eb[e^{-\theta(2M-t)}]$, we can also get for all $t>0$ and $0\le u\le \sqrt{t}$
\begin{align}\label{eq:conc2}
\Pb(t- 2M> u\sqrt{t}) \le c_5 e^{-u^2/4}.
\end{align}
Combining  \eqref{eq:conc1} and \eqref{eq:conc2}, we obtain \eqref{eq:concentration} with $c_3=2c_5$.

We now make use of the inequality
\begin{align}\label{eq:bound1}
\sup_{0\le s\le 1} \big |S_{st} - \sqrt{t} B^M_s \big| \le \sup_{0\le s\le 1} \big |S_{st} - \sqrt{2M} B^M_s \big| +\sup_{0\le s\le 1} \big | (\sqrt{2M} -\sqrt{t}) B^M_s \big|.
\end{align}
To bound the first term in the RHS of \eqref{eq:bound1}, note that it is equal to $0$ when $M=0$. Therefore for $t>1$, we have for all $c(\eta)>0$
\begin{align*}
&\Pb \bigg[ \sup_{0\le s\le 1} \big |S_{st} - \sqrt{2M} B^M_s \big| > c(\eta) \log t \bigg]\\
\le& 
\Pb \bigg[ \sup_{0\le s\le 1} \big |S_{st} - \sqrt{2M} B^M_s \big| > c(\eta) \log t; \, 0< M\le t \bigg] + \Pb (M>t).
\end{align*}
By \eqref{eq:B_M}, we deduce that for all $\eta>0$, there exists $c(\eta)>0$ such that for all $t>1$,
\begin{align*}
\Pb \bigg[ \sup_{0\le s\le 1} \big |S_{st} - \sqrt{2M} B^M_s \big| > c(\eta)\log t; \, 0<M\le t \bigg] < c(\eta) t^{-\eta}.
\end{align*}
On the other hand, applying \eqref{eq:concentration} for $u=\sqrt{t}$, we get $ \Pb (M>t) \le c_3 e^{-t/4}$.
By adjusting the value of $c(\eta)$, we have for all $t>0$ (the case $t\in (0,1]$ is covered by possibly enlarging $c(\eta)$)
\begin{align}\label{eq:rhs1}
\Pb \bigg[ \sup_{0\le s\le 1} \big |S_{st} - \sqrt{2M} B^M_s \big| > c(\eta) \log t \bigg] <c(\eta) t^{-\eta}.
\end{align}
To bound the second term in the RHS of \eqref{eq:bound1}, we apply \eqref{eq:concentration} for $u=2\sqrt{\eta} \sqrt{\log t}$. For any $\eta>0$, there is $t_0(\eta)>1$ such that the condition $u\le \sqrt{t}$ holds for $t\ge t_0(\eta)$. In \eqref{eq:combine1} and~\eqref{eq:combine2} below, we only consider $t\ge t_0(\eta)$.
We get 
\begin{align}\label{eq:combine1}
\Pb(|\sqrt{2M} -\sqrt{t}| > 2\sqrt{\eta} \sqrt{\log t})\le \Pb(|2M-t| > 2 \sqrt{\eta t\log t}) \le c_3 t^{-\eta}.
\end{align}
For the Brownian bridge $B^M$, see e.g.\ \cite[(2.2.22)]{MR838963}
\begin{align}\label{eq:bbg}
\Pb\bigg(\sup_{0\le s\le 1} |B^M_s| >z \bigg)=2\sum_{k=1}^\infty (-1)^{k-1} e^{-2k^2 z^2} \le 2e^{-2z^2}.
\end{align}
Taking $z=2\sqrt{\eta} \sqrt{\log t}$ in \eqref{eq:bbg}, we get
\begin{align}\label{eq:combine2}
\Pb\left(\sup_{0\le s\le 1} |B^M_s| > 2\sqrt{\eta}\sqrt{\log t} \right) \le 2t^{-8\eta}.
\end{align}
Combining \eqref{eq:combine1} and \eqref{eq:combine2}, there exists $c_6>0$, such that for all $t>0$ (the case $t<t_0(\eta)$ is covered by possibly enlarging $c_6$)
\begin{align}\label{eq:rhs2}
\Pb \left( \sup_{0\le s\le 1} \big | (\sqrt{2M} -\sqrt{t}) B^M_s \big| > 4\eta \log t\right) \le c_6 t^{- \eta}.
\end{align} 
Combining \eqref{eq:bound1}, \eqref{eq:rhs1} and \eqref{eq:rhs2} completes the proof.
\end{proof}

We are now ready to prove Theorem~\ref{lem:kmt3}.
\begin{proof}[Proof of Theorem~\ref{lem:kmt3}]
Consider a continuous-time random walk bridge $\gamma=(S_t, 0\le t\le t_\gamma) $ on $\Zb^d$.
Let $e_1, \cdots, e_d$ be the $d$ unit vectors in $\Rb^d$. Then $S_t = S_{t/d}^1\, e_1 + \cdots S^d_{t/d} \, e_d$ for all $0\le t\le t_\gamma$, where $S^1_t, \cdots, S^d_t$ are independent continuous-time random walk bridges on $\Zb$ with time length $t_\gamma/d$.
Theorem~\ref{lem:kmt3} then directly follows from Lemma~\ref{thm:kmt2} applied to $S^1_t, \cdots, S^d_t$.
\end{proof}

\section{Coupling the loop soups}

We now construct couplings between the random walk loop soups and the Brownian loop soups, thereby proving Theorems~\ref{thm1} and~\ref{thm2}.
We mostly follow the framework of \cite{LTF2007}, but make some subtle modifications, following a key observation in Section~\ref{sec:consbl}.

We first treat, in Sections~\ref{sec:prep}, \ref{sec:consbl} and~\ref{sec:pf}, the case of continuous-time random walk loop soups. 
We omit the adjective ``continuous-time'' in our description of random walk loop soups in these sections.
We then explain, in Section~\ref{subsec:dtrwls}, how to adapt this proof to the case of discrete-time random walk loop soups. 

\subsection{Preparation}\label{sec:prep}
In this subsection, we define some auxiliary random processes and loops, and construct a family of the random walk loop soups $\{\wt \Ac_\lambda\}_{\lambda>0}$, all on the same probability space $(\Omega, \Fc, \Pb)$.

We will construct the random walk loops with time length at least $2$, and then add independently an increasing family of random walk loop soups $\{\wt\Lc_\lambda\}_{\lambda>0}$ of loops with time length less than $2$.
For integers $n\ge1$ and $z\in \Zb^d$, let $N(n,z;t)$ be independent Poisson processes (parametrized by $t$) with parameter $Q_d(n)$ given by \eqref{eq:qdn}. 
For each $n\ge 1$, $m\ge 1, z\in \Zb^d$ we define the independent random variables $\wt T(n,z;m)$ with density
\begin{align*}
Q_d(n)^{-1} q_d(t) \mathbf{1}_{2n\le t <2n+2} dt.
\end{align*}
Let $\wt \ell(n,z;m)$ be a random walk bridge (loop) on $\Zb^d$ with time length $\wt T(n,z;m)$.
We make the loops $\wt \ell(n,z;m)$ for different $n,z;m$ independent from each other.
Then the collection of loops
\begin{align*}
\wt\Ac_\lambda:=\wt\Lc_\lambda \cup \big\{\wt\ell(n,z,m)+z : n\ge 1, z\in \Zb^d, 1\le m \le N(n, z; \lambda) \big\}
\end{align*}
is a random walk loop soup with intensity $\lambda$. It is increasing in $\lambda$ by construction.

Due to Theorem~\ref{lem:kmt3}, we can also construct on the same probability space $(\Omega, \Fc, \Pb)$,  i.i.d.\ Brownian loops $\ell(n,z;m), n\ge 1, z\in \Zb^d, m\ge 1$, distributed as $d$-dimensional Brownian bridges started at $0$ with time length $1$, so that the following holds. By Theorem~\ref{lem:kmt3}, there exists a constant $c(\eta, d)>0$,  such that for all $n\ge 1, z\in \Zb^d, m\ge 1$,
\begin{align}\label{eq:couple}
\Pb\bigg[\sup_{0\le s\le 1} \big | \wt \ell(n,z;m)(s \wt T) -  \big(\wt T/d\big)^{1/2}\ell(n,z; m) (s) \big| \ge c(\eta, d)   \log n \bigg] \le  c(\eta, d) n^{-\eta},
\end{align}
where we write $\wt T=\wt T(n,z;m)$ for brevity.

\subsection{Construction of the Brownian loop soup}\label{sec:consbl}
We now construct a Brownian loop soup  on the same probability space $(\Omega, \Fc, \Pb)$, using the random processes and loops defined in the previous subsection.
The main difference with the proof in \cite{LTF2007} is that we introduce an increasing sequence $\{a_n\}_{n\ge 1}$, defined by the following recurrence relation
\begin{align}\label{eq:an}
a_1:=(2\pi)^{-1}  \bigg ( (d/2) \int_2^\infty q_d(t) dt \bigg)^{-2/d}, \quad a_n^{-d/2} - a_{n+1}^{-d/2}= (2\pi)^{d/2}  (d/2)Q_d(n).
\end{align}
Note that $q_d(t)$ is integrable at $\infty$ by \eqref{eqnd}, so $a_1>0$ and the sequence is well defined.
This definition ensures that (recall that $(2\pi)^{-d/2} s^{-d/2-1}ds$ is the density of the time length of the Brownian loops in \eqref{eq:bl})
\begin{align}\label{eq:anint}
Q_d(n) =\int_{a_n}^{a_{n+1}} \frac{ds}{(2\pi)^{d/2} s^{d/2+1}}, 
\end{align}
hence we can couple the random walk loops with length in $(2n, 2n+2)$ rooted at $z$, and the Brownian loops with length in $(a_n, a_{n+1})$ rooted in a unit hypercube centered at $z$, in an exact one-to-one correspondence (i.e.\ there are the same number of loops in each set). 

In order to make our coupling work, we need to first show the following estimate. 

\begin{lemma}\label{lem:key}
For the sequence $\{a_n\}_{n\ge 1}$ defined in \eqref{eq:an}, we have
\[a_n =2n/d  +O(1).\]
\end{lemma}

\begin{proof}
By \eqref{eq:an}, for all $n\ge 2$,
\begin{align*}
a_n^{-d/2} = a_1^{-d/2} - \sum_{k=1}^{n-1} (2\pi)^{d/2} (d/2)Q_d(k),
\end{align*}
which tends to $0$ as $n\to \infty$, by our choice of $a_1$. This implies $a_n\to \infty$ as $n\to \infty$.

We also have for $m\ge n\ge 1$,
\begin{align}\label{eq:anm}
a_n^{-d/2} - a_m^{-d/2} =\sum_{k=n}^{m-1} (2\pi)^{d/2} (d/2)Q_d(k) =  (2\pi)^{d/2} (d/2) \int_{2n}^{2m} q_d(t) dt.
\end{align}
Letting $m\to \infty$, the left hand-side of \eqref{eq:anm} tends to $a_n^{-d/2}$, and the right hand-side of \eqref{eq:anm} tends to $ d^{d/2}(2n)^{-d/2} (1+O(n^{-1}))$ by \eqref{eqnd}. 
This leads to
\begin{align*}
a_n^{-d/2} =   d^{d/2}(2n)^{-d/2} (1+O(n^{-1})).
\end{align*}
It then implies
$a_n =(2n/d)(1+O(n^{-1}))$, which completes the proof.
\end{proof}

We now construct the Brownian loop soup on the same probability space $(\Omega, \Fc, \Pb)$  where the random walk loop soup in Section~\ref{sec:prep} lives.
We will only construct the loops with time length at least $a_1$, and then add independently an increasing family of Brownian loop soups $\{\Lc_\lambda\}_{\lambda>0}$ of loops with time length less than $a_1$. These  microscopic loops will not be coupled with the random walk loops.

Recall the Poisson processes $N(n, z;t)$ and the coupled loops $\wt\ell(n, z; m),  \ell(n,z;m)$ defined in Section~\ref{sec:prep}.
We further define independent  $\Rb^d$-valued random variables $Y(n,z;m)$ uniformly distributed on the hypercube $ \big\{\sum_{j=1}^d x_j e_j, |x_j| \le 1/2 \big\}$, and independent real random variables $T(n,z;m)$ with density
\begin{align}\label{eq:time}
(2\pi)^{-d/2} Q_d(n)^{-1} s^{-d/2-1} \mathbf{1}_{ a_n \le s\le a_{n+1}}.
\end{align}
Indeed, \eqref{eq:time} is the density of a probability measure by \eqref{eq:anint}.
We construct the rooted Brownian loop soup as follows
\begin{itemize}
\item Let $N(n, z;t)$ be the number of rooted loops that have appeared by time $t$ whose root is in the unit hypercube centered at $z$ and whose time duration is between $a_n$ and $a_{n+1}$;
\item Let $\ell^*(n,z; m)$ be the bridge obtained by scaling and translating $\ell(n,z,m)$, so that it has time duration $T(n,z; m)$ and root $z+Y(n,z; m)$.
\end{itemize}
Then the collection of loops
\begin{align*}
\Ac_\lambda:=\Lc_\lambda \cup \{\ell^*(n,z,m) : n\ge 1, z\in \Zb^d, 1\le m \le N(n, z; \lambda)\}
\end{align*}
is a Brownian loop soup with intensity $\lambda$. It is increasing in $\lambda$ by construction.

\subsection{Proof of Theorem~\ref{thm1}}\label{sec:pf}
We will show that the loop soups $\{\wt \Ac_\lambda \}_{\lambda>0}$ and $\{\Ac_\lambda \}_{\lambda>0}$ constructed in the previous subsections satisfy the conditions of Theorem~\ref{thm1}. 
Compared with Theorem~\ref{thm0}, we will make one more (small) improvement.
We obtain $c\lambda r^d$ as coefficient in the error probability which tends to $0$ as $\lambda \to 0$ (compared with the coefficient $c(\lambda+1) r^2$ in Theorem~\ref{thm0}).

\begin{proof}[Proof of Theorem~\ref{thm1}]
We fix $a>0$ and $\theta \in (0,2)$. The constants $c_0,c_1,\ldots$ that we will use in this proof are allowed to depend on $a,\theta, d$, but not on $r, \lambda, N$.
Fix 
\begin{equation}\label{eq:k}
k= 2+2a/d.
\end{equation}
For $r\ge 1$, $\lambda>0$ and integer $N\ge 1$, let
\begin{align*}
Z:=\sum_{|z|\le rN} \sum_{n\ge N^k} N(n,z;\lambda).
\end{align*}
Then $Z$ is a Poisson random variable with 
\begin{align*}
\Eb[Z]=\lambda \sum_{|z|\le rN} \sum_{n\ge N^k} Q_d(n) = \lambda \sum_{|z|\le rN} \int_{2\lceil N^k \rceil}^{\infty} q_d(t) dt
\le c_1 \lambda r^d N^d N^{-kd/2},
\end{align*}
where $c_1$ is a constant depending possibly on $a,d$, following \eqref{eqnd}.
This implies
\begin{align}\label{eq:p1}
\Pb\left[N(n,z; \lambda) >0 \text{ for some } n\ge N^k, |z| \le rN \right] \le c_1 \lambda r^d N^{(1-k/2)d} = c_1\lambda r^d N^{-a},
\end{align}
where the last equality is due to \eqref{eq:k}.

Let us denote the loops $\ell^*(n, z;m)$ and $\wt \ell(n,z;m)+z$ by $\gamma_{n,z,m}$ and $\wt \gamma_{n,z,m}$. Recall that they have respectively time length $T(n,z;m)$ and $\wt T(n,z;m)$, which we denote by $T ,\wt T$ for brevity.
Fix
\begin{align}\label{eq:eta}
\eta =(a+ d)/\theta. 
\end{align}
Recall the constant $c(\eta,d)$ in \eqref{eq:couple}.
We define the following event
\begin{align*}
A:=\bigg\{ \sup_{0\le s\le 1} &|\gamma_{n,z,m}(sT) - \wt \gamma_{n,z,m}(s \wt T)| \ge 3k \,c(\eta,d) \log N\\
 &\text{ for some } |z| <rN,  N^{\theta} < n < N^k, 1\le m\le N(n,z;\lambda) \bigg\}.
\end{align*}
Since
\begin{align*}
|\gamma_{n,z,m}(s T) - \wt \gamma_{n,z,m}(s\wt T)| 
\le &\big| (\wt \gamma_{n,z,m}(s \wt T)-z) - (\wt T/d)^{1/2} T^{-1/2} (\gamma_{n,z,m}(s T)-z - Y(n,z;m)) \big|\\
&+ \big| \big( (\wt T/d)^{1/2} T^{-1/2}-1\big) \big(\gamma_{n,z,m}(sT) -z  - Y(n,z;m) \big) \big|\\
&+ |Y(n,z;m)|,
\end{align*}
noting that $| Y(n,z;m)|\le \sqrt{d}/2$, we have 
\begin{align*}
A\subseteq\left\{A^1_{n,z,m} \cup A^2_{n,z,m} \text{ for some } |z| <rN,  N^{\theta} < n < N^k, 1\le m\le N(n,z;\lambda) \right\},
\end{align*}
where
\begin{align*}
&A^1_{n,z,m}=\bigg\{\sup_{0\le s\le 1} \big|  (\wt \gamma_{n,z,m}(s\wt T) -z)  -(\wt T/ d)^{1/2}T^{-1/2} (\gamma_{n,z,m}(s T) -z- Y(n,z;m)) \big| \ge c(\eta,d)  \log N^k \bigg\}, \\
&A^2_{n,z,m}=\bigg\{ \big| (\wt T/d)^{1/2} T^{-1/2}- 1 \big| \sup_{0\le s\le 1}|  \gamma_{n,z,m}(s T) -z- Y(n,z;m))| \ge c(\eta,d) \log N^k\bigg \}.
\end{align*}
Note that the rescaled loop $T^{-1/2} (\gamma_{n,z,m}(s T)-z- Y(n,z;m)), 0\le s\le 1$ is by definition $\ell(n,z;m)$, which is a Brownian bridge of time length $1$.
By \eqref{eq:couple}, for all $N^\theta < n < N^k$, there exists a constant $c_0$ depending only on $a, \theta, d$ (recall that $\eta$ is a function of $a,\theta$, according to \eqref{eq:eta}), such that 
\begin{equation}\label{eq:a2}
\Pb(A^1_{n,z,m}) \le c_0 n^{-\eta} \le c_0 N^{-\theta\eta} = c_0 N^{-a-d}.
\end{equation}
On the other hand, since $2n\le  \wt T<2n+2$,
$a_n \le T \le a_{n+1}$, and $a_n =2n/d +O(1)$ by Lemma~\ref{lem:key}, we have
\begin{align*}
\big|(\wt T/d)^{1/2}T^{-1/2} - 1\big| =O( n^{-1}).
\end{align*}
Therefore, there exist $c_3, c_4, c_5>0$, such that  for all $ N^\theta < n < N^k$,
\begin{align}\label{eq:a1}
\Pb(A^2_{n,z,m})\le \Pb\bigg [\sup_{0\le s\le 1}  | \gamma_{n,z,m}(s T) -z - Y(n,z;m)| \ge c_3 n \log n \bigg] \le \exp(- c_4n) \le c_5 N^{-a-d},
\end{align}
where the second inequality follows from the fact that $ \gamma_{n,z,m}(s T) -z - Y(n,z;m)$ is a 
$d$-dimensional Brownian bridge with time length $T\sim 2n/d$, and that  the maximum $M$ of a one-dimensional Brownian bridge with time length $1$ satisfies $\Pb(M > x) =\exp(-2x^2)$.

Now, let
\begin{align*}
W:=\sum_{|z|\le rN} \sum_{N^\theta<n< N^k} N(n,z;\lambda).
\end{align*}
Then  $W$ is a Poisson random variable with 
\begin{align*}
\Eb[W]\le \lambda \sum_{|z|\le rN} \sum_{n\ge1} Q_d(n) \le c_6(d) \lambda r^d N^{d},
\end{align*}
where $c_6(d) <\infty$ depends only on $d$.
Let $\Fc_0$ be the $\sigma$-algebra generated by $N(n,z;t)$ for all $n\ge 1, z\in\Zb^d, t>0$. Note that $W$ is measurable with respect to $\Fc_0$.

Combining \eqref{eq:a2} and \eqref{eq:a1}, we get
\begin{align*}
\Pb(A \mid \Fc_0) \le W (c_0+c_5) N^{-a-d},
\end{align*}
so there exists $c_7>0$ such that 
\begin{align}\label{eq:pa}
\Pb(A)\le \Eb(W) (c_0+c_5) N^{-a-d}  \le c_7\lambda r^d N^{-a}.
\end{align}
On the event $A^c \cap \{N(n,z; \lambda) =0 \text{ for all } n\ge N^k, |z| \le rN\}$ (whose complement has probability bounded by the sum of \eqref{eq:p1} and \eqref{eq:pa}), the coupling satisfies the conditions of Theorem~\ref{thm1}, with $c=\max \{c_1+c_7, \, 3k c(\eta, d)\}$. Indeed, \eqref{c2} is ensured by $A^c$, and \eqref{c1} is due to  \eqref{eq:c'}.
\end{proof}

\subsection{Proof of Theorem~\ref{thm2}}\label{subsec:dtrwls}
The proof for discrete-time random walk loop soups is similar, so we will only outline the differences. 
For simplicity, we keep the same notations as for the continuous-time random walk loop soups, but update some of their definitions.

First, we do not need to add the loop soups $\{\wt\Lc_\lambda\}_{\lambda>0}$ as in Section~\ref{sec:prep}, since there are by definition no loop with time length less than $2$ in the discrete-time random walk loop soup. For integers $n\ge1$ and $z\in \Zb^d$, we similarly define independent Poisson processes $N(n,z;t)$, but with parameter $\wt q_d(n)$ given by \eqref{eq:qn1}, instead of $Q_d(n)$.
For each $n\ge 1$, $m\ge 1, z\in \Zb^d$, let $\wt \ell(n,z;m)$ be an independent random walk bridge (loop) on $\Zb^d$ with time length $2n$.
Then the collection of loops
\begin{align*}
\wt\Ac_\lambda:= \big\{\wt\ell(n,z,m)+z : n\ge 1, z\in \Zb^d, 1\le m \le N(n, z; \lambda) \big\}
\end{align*}
is a random walk loop soup with intensity $\lambda$. It is increasing in $\lambda$ by construction.

For $d=1,2$, due to Theorem~\ref{thm:kmt1} and Lemma~\ref{lem:kmt}, we can also construct on the same probability space $(\Omega, \Fc, \Pb)$,  i.i.d.\ Brownian loops $\ell(n,z;m), n\ge 1, z\in \Zb^d, m\ge 1$, distributed as $d$-dimensional Brownian bridges started at $0$ with time length $1$, so that the following holds. 
\begin{align}\label{eq:couple-12}
\Pb\bigg[\sup_{0\le s\le 1} \big | \wt \ell(n,z;m) (2ns) -  (2n/d)^{1/2}\ell(n,z,m)(s) \big| \ge c(\eta, d)   \log n \bigg] \le  c(\eta, d) n^{-\eta}.
\end{align}
For $d\ge 3$, by Lemma~\ref{lem:kmtd}, we get a slightly weaker bound, namely there exist $c(\eta, d)$,  such that for all $n\ge 1, z\in \Zb^d, m\ge 1$,
\begin{align}\label{eq:couple-d}
\Pb\bigg[\sup_{0\le s\le 1} \big | \wt \ell(n,z;m) (2ns) -  (2n/d)^{1/2}\ell(n,z,m)(s) \big| \ge c(\eta, d) n^{1/4}  \log n \bigg] \le  c(\eta, d) n^{-\eta}.
\end{align}
The bounds \eqref{eq:couple-12} and \eqref{eq:couple-d} are the analogues of \eqref{eq:couple}.

We define the sequence $\{a_n\}_{n\ge 1}$ by the following recurrence relation
\begin{align}\label{eq:an12}
a_1:=(2\pi)^{-1} \bigg((d/2) \sum_{n=1}^\infty \wt q_d(n) \bigg)^{-2/d}, \quad a_n^{-d/2} - a_{n+1}^{-d/2}= (2\pi)^{d/2} (d/2) \wt q_d(n).
\end{align}
We can then similarly prove that $a_n = 2n/d +O(1).$
To construct the Brownian loop soup, we proceed exactly as in Section~\ref{sec:consbl}, but replace $Q_d(n)$ by $\wt q_d(n)$ in \eqref{eq:time}. 

For $d=1,2$, the rest of the proof is exactly the same as in Section~\ref{sec:pf}.
For $d\ge 3$, we still need to update the definitions of the events $A$, $A^1_{n,z,m}$ and $A^2_{n,z,m}$ to
\begin{equation}\label{eq:A}
\begin{split}
A:=\bigg\{ \sup_{0\le s\le 1} &|\gamma_{n,z,m}(sT) - \wt \gamma_{n,z,m}(s \wt T)| \ge 3k \,c(\eta,d) N^{k/4} \log N\\
 &\text{ for some } |z| <rN,  N^{\theta} < n < N^k, 1\le m\le N(n,z;\lambda) \bigg\}
 \end{split}
\end{equation}
and
\begin{align*}
&A^1_{n,z,m}=\bigg\{\sup_{0\le s\le 1} \big|  (\wt \gamma_{n,z,m}(s\wt T) -z)  -(\wt T/ d)^{1/2}T^{-1/2} (\gamma_{n,z,m}(s T) -z - Y(n,z;m)) \big| \ge c(\eta,d) N^{k/4} \log N^k \bigg\}, \\
&A^2_{n,z,m}=\bigg\{ \big( (\wt T/d)^{1/2} T^{-1/2}- 1 \big) \sup_{0\le s\le 1}|  \gamma_{n,z,m}(s T) -z - Y(n,z;m)| \ge c(\eta,d)N^{k/4} \log N^k\bigg \}.
\end{align*}
By \eqref{eq:couple-d}, we deduce that for all $N^\theta < n < N^k$, there exists a constant $c_0$ depending only on $a, \theta, d$, such that $\Pb(A^1_{n,z,m}) \le c_0 N^{-a-d}$. We get a similar bound for $\Pb(A^2_{n,z,m})$, as in \eqref{eq:a1}. 
The definition \eqref{eq:A} of $A$ implies the bound \eqref{c33} for the  loops which are rescaled by $N^{-1}$ in space, noting that $k=2+2a/d$ by \eqref{eq:k}.

\subsection*{Acknowledgments}
We thank Yifan Gao, Xinyi Li and Pierre Nolin for useful discussions.
We thank the anonymous referees for their careful reading and for pointing out a citation mistake related to the KMT coupling in an earlier version of the manuscript.
We also acknowledge the support by a GRF grant from the Research Grants Council of the Hong Kong SAR (project 11305825).

\bibliographystyle{abbrv}
\bibliography{cr}

\begin{thebibliography}{10}

\bibitem{ABJL2023}
E.~A\"{\i}d\'{e}kon, N.~Berestycki, A.~Jego, and T.~Lupu.
\newblock Multiplicative chaos of the {B}rownian loop soup.
\newblock {\em Proc. Lond. Math. Soc. (3)}, 126(4):1254--1393, 2023.

\bibitem{ALS2020b}
J.~Aru, T.~Lupu, and A.~Sep\'{u}lveda.
\newblock The first passage sets of the 2{D} {G}aussian free field: convergence
  and isomorphisms.
\newblock {\em Comm. Math. Phys.}, 375(3):1885--1929, 2020.

\bibitem{MR2137045}
K.~Ball and J.~Sterbenz.
\newblock Explicit bounds for the return probability of simple random walks.
\newblock {\em J. Theoret. Probab.}, 18(2):317--326, 2005.

\bibitem{CD2024}
Z.~Cai and J.~Ding.
\newblock {Incipient infinite clusters and volume growth for Gaussian free
  fields and loop soups on metric graphs}.
\newblock {\em arXiv preprint arXiv:2412.05709}, 2024.

\bibitem{CD20242}
Z.~Cai and J.~Ding.
\newblock One-arm probabilities for metric graph {G}aussian free fields below
  and at the critical dimension.
\newblock {\em arXiv preprint arXiv:2406.02397}, 2024.

\bibitem{drewitz2025critical}
A.~Drewitz, A.~Pr{\'e}vost, and P.-F. Rodriguez.
\newblock {Critical one-arm probability for the metric Gaussian free field in
  low dimensions}.
\newblock {\em Probability Theory and Related Fields}, pages 1--24, 2025.

\bibitem{GNQ1}
Y.~Gao, P.~Nolin, and W.~Qian.
\newblock {Percolation of discrete GFF in dimension two I. Arm events in the
  random walk loop soup}.
\newblock {\em arXiv preprint arXiv:2409.16230}, 2024.

\bibitem{GNQ2}
Y.~Gao, P.~Nolin, and W.~Qian.
\newblock {Percolation of discrete GFF in dimension two II. Connectivity
  properties of two-sided level sets}.
\newblock {\em arXiv preprint arXiv:2409.16273}, 2024.

\bibitem{GNQ3}
Y.~Gao, P.~Nolin, and W.~Qian.
\newblock Up-to-constants estimates on four-arm events for simple conformal
  loop ensemble.
\newblock {\em arXiv preprint arXiv:2504.06202}, 2025.

\bibitem{MR375412}
J.~Koml\'os, P.~Major, and G.~Tusn\'ady.
\newblock An approximation of partial sums of independent {${\rm RV}$}'s and
  the sample {${\rm DF}$}. {I}.
\newblock {\em Z. Wahrscheinlichkeitstheorie und Verw. Gebiete}, 32:111--131,
  1975.

\bibitem{MR2677157}
G.~F. Lawler and V.~Limic.
\newblock {\em Random walk: a modern introduction}, volume 123 of {\em
  Cambridge Studies in Advanced Mathematics}.
\newblock Cambridge University Press, Cambridge, 2010.

\bibitem{LTF2007}
G.~F. Lawler and J.~A. Trujillo~Ferreras.
\newblock Random walk loop soup.
\newblock {\em Trans. Amer. Math. Soc.}, 359(2):767--787, 2007.

\bibitem{MR2045953}
G.~F. Lawler and W.~Werner.
\newblock The {B}rownian loop soup.
\newblock {\em Probab. Theory Related Fields}, 128(4):565--588, 2004.

\bibitem{MR2815763}
Y.~Le~Jan.
\newblock {\em Markov paths, loops and fields}, volume 2026 of {\em Lecture
  Notes in Mathematics}.
\newblock Springer, Heidelberg, 2011.
\newblock Lectures from the 38th Probability Summer School held in Saint-Flour,
  2008, {\'E}cole d'{\'E}t{\'e} de Probabilit{\'e}s de Saint-Flour.

\bibitem{MR3502602}
T.~Lupu.
\newblock From loop clusters and random interlacements to the free field.
\newblock {\em Ann. Probab.}, 44(3):2117--2146, 2016.

\bibitem{Lu2019}
T.~Lupu.
\newblock Convergence of the two-dimensional random walk loop-soup clusters to
  {CLE}.
\newblock {\em J. Eur. Math. Soc. (JEMS)}, 21(4):1201--1227, 2019.

\bibitem{lupu2018random}
T.~Lupu and W.~Werner.
\newblock The random pseudo-metric on a graph defined via the zero-set of the
  {G}aussian free field on its metric graph.
\newblock {\em Probab. Theory Related Fields}, 171(3-4):775--818, 2018.

\bibitem{LW2025}
T.~Lupu and W.~Werner.
\newblock Intensity doubling for {B}rownian loop-soups in high dimensions.
\newblock {\em arXiv preprint arXiv:2511.21670}, 2025.

\bibitem{QW2019}
W.~Qian and W.~Werner.
\newblock Decomposition of {B}rownian loop-soup clusters.
\newblock {\em J. Eur. Math. Soc. (JEMS)}, 21(10):3225--3253, 2019.

\bibitem{MR3877544}
A.~Sapozhnikov and D.~Shiraishi.
\newblock On {B}rownian motion, simple paths, and loops.
\newblock {\em Probab. Theory Related Fields}, 172(3-4):615--662, 2018.

\bibitem{MR2979861}
S.~Sheffield and W.~Werner.
\newblock Conformal loop ensembles: the {M}arkovian characterization and the
  loop-soup construction.
\newblock {\em Ann. of Math. (2)}, 176(3):1827--1917, 2012.

\bibitem{MR838963}
G.~R. Shorack and J.~A. Wellner.
\newblock {\em Empirical processes with applications to statistics}.
\newblock Wiley Series in Probability and Mathematical Statistics: Probability
  and Mathematical Statistics. John Wiley \& Sons, Inc., New York, 1986.

\bibitem{BCL2016}
T.~van~de Brug, F.~Camia, and M.~Lis.
\newblock Random walk loop soups and conformal loop ensembles.
\newblock {\em Probab. Theory Related Fields}, 166(1-2):553--584, 2016.

\bibitem{Werner2025}
W.~Werner.
\newblock {A switching identity for cable-graph loop soups and Gaussian free
  fields}.
\newblock {\em arXiv preprint arXiv:2502.06754}, 2025.

\end{thebibliography}

\end{document}